\documentclass{amsart}

\usepackage{amsmath,amssymb,amsthm}
\usepackage{pinlabel}

\newtheorem{prop}{Proposition}
\newtheorem{thm}[prop]{Theorem}
\newtheorem{lem}[prop]{Lemma}
\newtheorem{cor}[prop]{Corollary}

\theoremstyle{definition}

\newtheorem{rem}[prop]{Remark}

\newtheorem*{ack}{Acknowledgements}

\def\co{\colon\thinspace}

\newcommand{\oB}{\overline{B}}

\newcommand{\C}{\mathbb{C}}

\newcommand{\D}{\mathbb{D}}
\newcommand{\rmd}{\mathrm{d}}

\newcommand{\rme}{\mathrm{e}}

\newcommand{\F}{\mathbb{F}}

\newcommand{\oG}{\overline{G}}

\newcommand{\bfh}{\mathbf{h}}

\newcommand{\rmi}{\mathrm{i}}

\newcommand{\wtM}{\widetilde{M}}
\newcommand{\wtm}{\widetilde{m}}

\newcommand{\wQ}{\widetilde{Q}}

\newcommand{\R}{\mathbb{R}}

\newcommand{\RP}{\mathbb{R}\mathrm{P}}

\newcommand{\bfs}{\mathbf{s}}
\newcommand{\Sb}{S_{\mathrm{b}}}

\newcommand{\bft}{\mathbf{t}}

\newcommand{\utsw}{u^{\bft}_{\bfs,w}}

\newcommand{\WW}{\mathcal{W}}

\newcommand{\bfx}{\mathbf{x}}
\newcommand{\xist}{\xi_{\mathrm{st}}}

\newcommand{\bfy}{\mathbf{y}}

\newcommand{\Z}{\mathbb{Z}}

\DeclareMathOperator{\ev}{ev}

\DeclareMathOperator{\im}{Im}
\DeclareMathOperator{\Int}{Int}

\DeclareMathOperator{\re}{Re}

\begin{document}

\author[H.~Geiges]{Hansj\"org Geiges}
\address{Mathematisches Institut, Universit\"at zu K\"oln,
Weyertal 86--90, 50931 K\"oln, Germany}
\email{geiges@math.uni-koeln.de}
\author[K.~Zehmisch]{Kai Zehmisch}
\address{Mathematisches Institut, WWU M\"unster,
Einstein\-stra\-{\ss}e 62, 48149 M\"unster, Germany}
\email{kai.zehmisch@uni-muenster.de}

\title{The Weinstein conjecture for connected sums}

\date{}

\begin{abstract}
We prove the Weinstein conjecture for non-trivial contact connected sums
under either of two topological conditions: non-trivial
fundamental group or torsion-free homology.
\end{abstract}


\thanks{H.~G.\ and K.~Z.\ are partially supported by DFG grants
GE 1245/2-1 and ZE 992/1-1, respectively.}

\maketitle

%%%%%%%%%%%%%%%%%%%%%%%%%%%%%%%%%%%%%%%%%%%%%%%%%%%%%%%%%%%%%%%%%%%%%%

\section{Introduction}
A $(2n+1)$-dimensional closed manifold $M$ is said to satisfy the
Weinstein conjecture if the Reeb vector field of any contact form on $M$
has a periodic orbit. For a recent survey on the status of this
conjecture see~\cite{pasq12}. In his seminal paper~\cite{hofe93}
on the Weinstein conjecture in dimension three, Hofer
proved the Weinstein conjecture for overtwisted $3$-manifolds
and for any closed, orientable $3$-manifold with non-vanishing
second homotopy group; in fact, in either of these cases
there exists a \emph{contractible} periodic Reeb orbit.

These results can be combined into the following statement.

\begin{thm}[Hofer]
\label{thm:Hofer}
Let $(M,\xi)=(M_1,\xi_1)\#(M_2,\xi_2)$ be the contact connected sum
of two closed, connected contact manifolds of dimension~$3$.
Suppose we can write  $\xi=\ker\alpha$ with a
contact form $\alpha$ whose Reeb vector field does not have
any contractible periodic orbits. Then one of the
summands $(M_i,\xi_i)$ is contactomorphic to the $3$-sphere $S^3$
with its standard tight contact structure~$\xist$.
\end{thm}

\begin{proof}
Consider the belt sphere $S^2\subset M$ of the connected sum.
Since there are no contractible Reeb orbits, the contact structure
$\xi$ must be tight, and by a perturbation of $S^2$ one may achieve
that the characteristic foliation of $S^2$ induced by $\xi$
has two elliptic singular points connected by all the leaves
of the foliation, cf.\ \cite[Section 4.6]{geig08}. Under the
assumption of the non-existence of contractible
periodic orbits, Hofer then produces
a filling of $S^2$ by holomorphic discs in the half-symplectisation
$(-\infty,0]\times M$, which projects to a $3$-cell in $M$
bounded by~$S^2$. Together with the
Poincar\'e conjecture and Eliashberg's classification
of tight contact structures on~$S^3$, one concludes that
one of the $(M_i,\xi_i)$ equals $(S^3,\xist)$.
\end{proof}

Our aim is to prove an analogous result in higher dimensions.

\begin{thm}
\label{thm:connected}
Let $(M,\xi)=(M_1,\xi_1)\#(M_2,\xi_2)$ be the contact connected sum of two
closed, connected contact manifolds of dimension $2n+1\geq 5$. Suppose we
can write $\xi=\ker\alpha$ with a
contact form $\alpha$ whose Reeb vector field does not have
any contractible periodic orbits. Assume further that
one of the following conditions is satisfied:
\begin{itemize}
\item[(i)] $M$ is simply connected and has torsion-free homology, or
\item[(ii)] $M$ is not simply connected.
\end{itemize}
Then one of the summands $M_1,M_2$ is homeomorphic to~$S^{2n+1}$.
\end{thm}

Put another way, if $(M,\xi)$ is a non-trivial connected sum
(i.e.\ neither summand is a homotopy sphere)
and $M$ satisfies condition (i) or (ii), then any contact
form defining $\xi$ will have a contractible periodic Reeb orbit.
This result has the potential to serve as a contact-geometric
primality test for manifolds. For instance, if $(Q,g)$
is a Riemannian manifolds without contractible periodic geodesics
(say, a manifold with non-positive sectional curvature),
then the unit cotangent bundle $ST^*Q$ with its canonical contact form
does not have any contractible periodic Reeb orbits, since the
Reeb flow equals the cogeodesic flow, cf.\ \cite[Section~1.5]{geig08}.
If $ST^*Q$ satisfies one of the topological assumptions (i)
or~(ii), it cannot be a non-trivial contact connected sum.

In their work on the planar circular restricted three-body problem,
Albers et al.~\cite{afkp12} show that the energy hypersurface
for energies slightly above the energy of the first Lagrange point
is the contact connected sum of two copies of $\RP^3$ with its
standard tight contact structure. Hofer's Theorem~\ref{thm:Hofer}
then implies the existence of a periodic orbit. Our
Theorem~\ref{thm:connected} may have similar applications.

\vspace{1mm}

The germ of the contact structure $\xi=\ker\alpha$ near a compact
hypersurface $S\subset M$ is determined by the $1$-form
$\alpha|_{TS}$, cf.\ \cite[Proposition~6.4]{dige12}.
In a contact connected sum as in~\cite{wein91}, cf.\ \cite[Chapter~6]{geig08},
the contact form along the belt sphere $S^{2n}$ equals
$\frac{1}{2}(\bfx\, \rmd\bfy-\bfy\, \rmd\bfx)|_{TS^{2n}}$,
where $S^{2n}$ is regarded as a subset of $\R\times\R^n\times\R^n$
with coordinates $(w,\bfx,\bfy)$. We shall call a $2n$-sphere
with a germ of a contact structure equal to this one
a {\bf standard} $2n$-sphere. Conversely, given a
\emph{separating} standard $2n$-sphere in a contact manifold, the manifold is
a contact connected sum.

The following theorem tells us that under the assumption on
the non-existence of contractible periodic Reeb orbits,
the property of $S^{2n}$ being separating follows automatically
from a knowledge of the contact structure near~$S^{2n}$.
So in Theorem~\ref{thm:connected} it would suffice to require the existence
of any standard $2n$-sphere.

\begin{thm}
\label{thm:sphere}
Let $(M,\xi=\ker\alpha)$ be a closed $(2n+1)$-dimensional contact manifold,
$n\geq 2$, such that $\alpha$ does not have any contractible periodic
Reeb orbits. If $S^{2n}\subset M$ is a standard sphere,
then $S^{2n}$ separates, so that $(M,\xi)$
is a contact connected sum.
\end{thm}

This immediately leads to the following corollary.

\begin{cor}
If a contact manifold $(M,\xi)$ is obtained by an index~$1$ surgery
on a closed, \emph{connected} contact manifold $(M',\xi')$,
then any contact form defining $\xi$ must have contractible
periodic Reeb orbits.
\end{cor}

\begin{proof}
The belt sphere of the $1$-handle corresponding to the surgery does
not separate~$M$. Now apply Theorem~\ref{thm:sphere}.
\end{proof}

For example, the standard contact structure on $S^1\times S^{2n}$
is obtained by an index~$1$ surgery on the standard $S^{2n+1}$.
So any contact form defining this contact
structure must have a contractible periodic Reeb orbit.

\vspace{1mm}

The strategy for proving Theorem~\ref{thm:sphere}
is to show that the moduli space of suitable holomorphic discs,
which will be compact under the assumption on the
non-existence of contractible periodic Reeb orbits,
leads to a filling of $S^{2n}$ inside~$M$, i.e.\ $S^{2n}$
is homologically trivial. With the additional assumption (i) or
(ii) in Theorem~\ref{thm:connected}, this filling can be
shown to be a ball.

We start with a moduli space $\WW$ of holomorphic discs with
Lagrangian boundary condition inside a half-symplectisation $W$
of~$(M,\xi)$. By deforming the evaluation map $\ev\co
\WW\times\D\rightarrow W$ to a map into $M$ which sends
the boundary $\partial(\WW\times\D)$ onto the belt sphere
of the index~$1$ surgery, we arrive at the desired
topological conclusions.

A different approach to the type of questions
addressed by Theorems \ref{thm:connected} and~\ref{thm:sphere}
has been suggested by Niederkr\"uger~\cite{nied13}
in collaboration with Ghiggini and Wendl.
\section{Contact surgery of index~$1$}
\subsection{The model handle}
Given a standard sphere $S^{2n}\subset (M,\xi)$ as assumed
in Theorem~\ref{thm:sphere}, a neighbourhood of it can be
identified with a neighbourhood of the belt sphere
in the standard picture for an index~$1$ surgery, see for instance
\cite[Figure~6.4]{geig08}. This allows one to perform a reverse contact
surgery on that sphere. In other words, $(M,\xi)$
is the result of performing an index~$1$ surgery on a suitable
contact manifold $(M',\xi')$.

The standard contact structure $\ker\bigl(\rmd w+\frac{1}{2}(\bfx\,\rmd\bfy-
\bfy\,\rmd\bfx)\bigr)$ on $\R^{2n+1}$ is preserved by the map
$(w,\bfx,\bfy)\mapsto(\lambda^2 w,\lambda\bfx,\lambda\bfy)$;
the homothety $(w,\bfx,\bfy)\mapsto(w,\bfx,\bfy)/R$ pulls back the
standard contact form to $\bigl(R\,\rmd w+
\frac{1}{2}(\bfx\,\rmd\bfy-\bfy\,\rmd\bfx)\bigr)/R^2$.
Therefore we may assume that the index~$1$ surgery on $(M',\xi')$ is performed
inside two balls of large radius $R$, with contact form
$\pm R\,\rmd w+\frac{1}{2}(\bfx\,\rmd\bfy-\bfy\,\rmd\bfx)$ for $\xi'$
on the respective ball.

Here is an explicit description of this surgery. We slightly modify
the picture in \cite{wein91} and \cite{geig08} so as to obtain
a good model for the holomorphic analysis.

\begin{figure}[h]
\labellist
\small\hair 2pt
\pinlabel $(w,\bfx,\bfy)$ [l] at 218 461
\pinlabel $v$ [t] at 460 215
\pinlabel $R$ [b] at 405 219
\pinlabel $-R$ [b] at 21 218
\pinlabel ${\text{flow of $Y$}}$ [tr] at 379 347
\pinlabel ${\text{upper boundary}}$ [l] at 78 352
\pinlabel ${\text{lower}}$ [l] at 397 310
\pinlabel ${\text{boundary}}$ [l] at 397 295
\endlabellist
\centering
\includegraphics[scale=0.52]{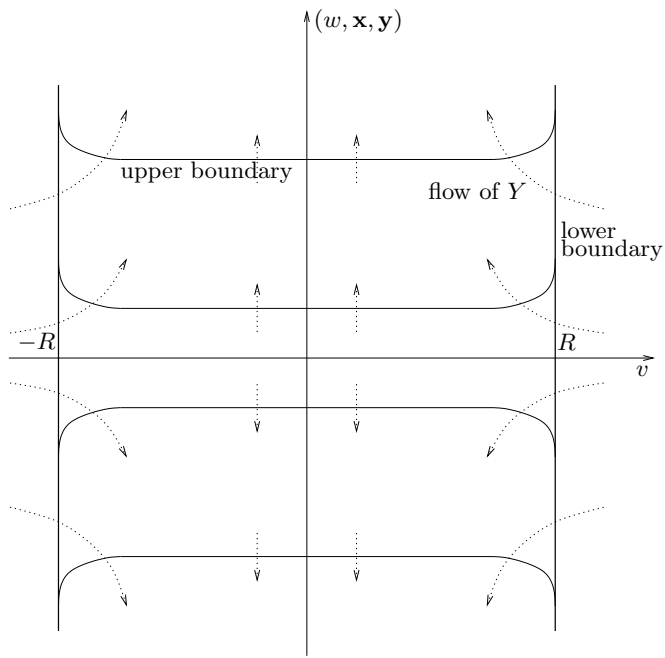}
  \caption{The surgery model}
  \label{figure:handle}
\end{figure}

On $\R^{2n+2}$ with coordinates $v,w\in\R$ and $\bfx,\bfy\in\R^n$
we consider the symplectic form $\omega_0=\rmd v\wedge\rmd w+
\rmd\bfx\wedge\rmd\bfy$ and the Liouville vector field
\[ Y_0=-v\,\partial_v+2w\,\partial_w+
\frac{1}{2}(\bfx\,\partial_{\bfx}+\bfy\,\partial_{\bfy}).\]
This vector field is transverse to the hypersurfaces
$\{v=\pm R\}$, where it induces the contact form
\[ i_{Y_0}\omega_0=\mp R\,\rmd w+\frac{1}{2}(\bfx\,\rmd\bfy-\bfy\,\rmd\bfx),\]
so we may identify balls of Radius $R$ in $\{v=\pm R\}$
with the two Darboux charts in $(M',\xi')$ where we want
to perform the surgery. This surgery is performed by replacing
two balls of radius $r<R$ in these charts with a cylinder
\[ \bigl\{(v,w,\bfx,\bfy)\co w^2+\frac{1}{2}(|\bfx|^2+|\bfy|^2)=r^2\bigr\} \]
and rounding corners, see
Figure~\ref{figure:handle}. That figure shows a `thick' handle
with $r$ close to~$R$, and a `thin' handle with $r=1$, say. Since $Y_0$ is
everywhere transverse to the new hypersurface, it induces a
contact structure on the manifold obtained by index~$1$ surgery.
On the `straight' part of the cylinder we glued in, i.e.\
away from a small region where we had to round corners,
we would like to replace $Y_0$ by the `radial' Liouville vector field
\[ Y_1=w\,\partial_w+\frac{1}{2}(\bfx\,\partial_{\bfx}+\bfy\,\partial_{\bfy}).\]
This $Y_1$ has the advantage of being the gradient
vector field of the strictly plurisubharmonic function
\[ (v,w,\bfx,\bfy)\longmapsto
\frac{1}{2}w^2+\frac{1}{4}(|\bfx|^2+|\bfy|^2)\]
on $\C^{n+1}$ with coordinates $v+\rmi w,\bfx+\rmi\bfy$,
so the induced contact structure on the cylinder is given by the tangent
hyperplanes invariant under the standard complex structure $\rmi$
on~$\C^{n+1}$.

It is indeed possible to interpolate between the Liouville
vector fields $Y_0$ near $v=\pm R$ and $Y_1$ for $|v|<R-\varepsilon$
via Liouville vector fields transverse to the boundary of the
handle. Let $v\mapsto a(v)$ be a smooth function
interpolating monotonically between $a\equiv -1$ for $|v|\leq
R-\varepsilon$ and $a\equiv 0$ for $|v|\geq R-\varepsilon/2$.
Then modify $Y_0$ by adding the Hamiltonian vector field $X_H$
of the function $H(v,w,\bfx,\bfy)=a(v)\cdot vw$. This will certainly
produce a Liouville vector field $Y$, and it remains to check
the other desired properties.

We have
\[ X_H=-\frac{\partial H}{\partial w}\,\partial_v+
\frac{\partial H}{\partial v}\,\partial_w
= -a(v)\cdot v\,\partial_v+\bigl(a'(v)\cdot vw+a(v)\cdot w\bigr)\,\partial_w,\]
hence
\begin{eqnarray*}
Y & = & Y_0+X_H\\
  & = & -\bigl(1+a(v)\bigr)\cdot v\,\partial_v+
        \bigl(2+a'(v)\cdot v+a(v)\bigr)\cdot w\,\partial_w+
        \frac{1}{2}(\bfx\,\partial_{\bfx}+\bfy\,\partial_{\bfy}).
\end{eqnarray*}
So $Y$ interpolates between $Y_0$ for $|v|\geq R-\varepsilon/2$
and $Y_1$ for $|v|\leq R-\varepsilon$ as desired. Moreover,
the coefficient of $w\,\partial_w$ is always positive,
which guarantees the transversality requirement.

A similar construction can be found in \cite[Theorem~7.6.2]{geig08}.

\subsection{The symplectic cobordism}
\label{subsection:cobordism}
Write $M_1$, $M_R$ for the copies of $M$ obtained by attaching the thin
resp.\ thick handle to~$M'$. We write $\alpha_1$ for a contact form on
$M_1$ inducing the contact
structure $\xi'$ away from the thin handle and equal to $i_Y\omega_0$
on the upper boundary of the thin handle, so that $(M_1,\ker\alpha_1)$
is contactomorphic to $(M,\xi)$. Similarly, we define $\alpha_R$ on~$M_R$.
Let $\alpha$ be any other contact form on $M$ with $\ker\alpha=\xi$.

We now build a symplectic manifold $(W,\omega)$ as follows.
Write $\alpha_1=\rme^g\alpha$ with a suitable smooth
function $g\co M\rightarrow\R$. Since $M$ is compact,
we can choose an $a_0\in\R^+$ such that $\rme^{-a_0}<\rme^g$. Then
\[ \bigl\{ (a,x)\in\R\times M\co -a_0\leq a\leq g(x)\bigr\}\]
with symplectic form $\rmd(\rme^a\alpha)$ and Liouville vector field
$\partial_a$ defines a cylindrical Liouville cobordism between
\[ \bigl(\{-a_0\}\times M,\rme^{-a_0}\alpha\bigr)\;\;\text{and}\;\;
(M_1,\alpha_1).\]
To the concave end $\{-a_0\}\times M$ of this symplectic cobordism, we
can attach the half-symplec\-ti\-sa\-tion
\[ \bigl((-\infty,-a_0]\times M,\rmd(\rme^a\alpha)\bigr).\]
Finally, to the convex end $M_1$ of the cylindrical Liouville
cobordism we attach the part of the handle between the
thin and the thick upper boundary.

The resulting symplectic manifold $(W,\omega)$ admits a globally
defined and nowhere vanishing Liouville vector field inducing the
appropriate contact form on the trans\-ver\-sals $\{-a_0\}\times M$, $M_1$,
and~$M_R$.
\subsection{The almost complex structure}
We choose an $\omega$-compatible almost complex structure
$J$ on $W$ with the following properties:
\begin{itemize}
\item[(J1)] On $(-\infty,-a_0]\times M$ we choose
an $\R$-invariant almost complex structure $J$ preserving
$\ker\alpha$ and sending $\partial_a$ to the Reeb vector field
of~$\alpha$.
\item[(J2)] On the part between the thin and the thick handle, i.e.\
between $M_1$ and~$M_R$, we choose $J=\rmi$ for $|v|\leq R-\varepsilon$,
and generic elsewhere, subject to
the condition that $M_R$ be a (strictly) $J$-convex boundary.
\item[(J3)] On $[-a_0,0]\times M$ we choose any generic almost complex
structure extending the choices in (J1) and~(J2).
\end{itemize}
\section{The moduli space of holomorphic discs}
\subsection{The Lagrangian boundary condition}
We are now going to describe a family of Lagrangian
cylinders in $\C^{n+1}$. Their intersection with
the part of $(W,\omega)$ between $M_1$ and $M_R$ ---
which for ease of reference we shall simply refer to as
the `handle' --- will provide the boundary condition for the
holomorphic discs we wish to study.

Recall that on $\C^{n+1}$ we have the complex coordinates
$v+\rmi w,\bfx+\rmi\bfy$; we now write the latter as
$\bfx+\rmi\bfy=\bfh+z_n$ with $\bfh\in\C^{n-1}$ and $z_n\in\C$.
For each $\bft\in\R^{n-1}$, the Lagrangian cylinder
\[ L^{\bft}:=\bigl(\{0\}\times\R_w\bigr)\times
\bigl(\R^{n-1}_{\re(\bfh)}\times\{\rmi\bft\}\bigr)\times
\partial D^2_{z_n}\subset\C\times\C^{n-1}\times\C\]
is filled by standard holomorphic discs
\[ \D\ni z\longmapsto\utsw(z):=\bigl(0,w;\bfs+\rmi\bft;z\bigr),\]
where $\D$ denotes the unit disc in $\C$, and the parameters
$w,\bfs$ range over $\R$ and $\R^{n-1}$, respectively.
\subsection{Definition of the moduli space}
We define $\WW$ to be the moduli space of holomorphic
discs~$u\co\D\rightarrow W$ (smooth up to the boundary) 
with $u(\partial\D)\subset L^{\bft}\cap W$, where $\bft$ is allowed
to vary over $\R^{n-1}$, and satisfying the following conditions:
\begin{itemize}
\item[(M1)] The relative homology class $[u]\in H_2(W,L^{\bft}\cap W)$
equals that of some standard disc $\utsw$ in $W$.
\item[(M2)] The parametrisation of $u$ is fixed by the
requirement $u(\rmi^k)\in L^{\bft}\cap\{z_n=\rmi^k\}$ for $k=0,1,2$.
\end{itemize}
Observe that standard discs of the
same level $\bft$ are homotopic in $(W,L^{\bft}\cap W)$.

\subsection{Properties of the holomorphic discs}
The present set-up is similar to that in our previous
paper~\cite{geze}, and the basic properties of the holomorphic discs
are established by the same arguments as in \cite[Section~3.3]{geze}.

\begin{prop}
The discs $u\in\WW$ have the following properties:
\begin{itemize}
\item[(i)] The Maslov index equals~$2$.
\item[(ii)] The symplectic energy equals~$\pi$.
\item[(iii)] All discs are simple.
\end{itemize}
\end{prop}
\subsection{Compactness}
The bubbling of spheres is impossible in an exact symplectic manifold.
Both bubbling of finite energy planes and breaking are prevented by the
existence of a global non-singular Liouville vector field
(and the resulting maximum principle) and the assumption that there be
no contractible periodic Reeb orbits. The bubbling of discs at
boundary points is ruled out by the Lagrangian boundary condition
(which fixes the energy) and the homotopical condition (M2).
This guarantees the compactness of~$\WW$. For more details
see~\cite[Section~4]{geze}.
\subsection{Transversality}
If a holomorphic disc $u\in\WW$ has its image $u(\D)$ contained entirely
in the part of the handle determined by $|v|\leq R-\varepsilon$,
where $J=\rmi$, then $u$ can be written globally in terms
of holomorphic component functions $(v+\rmi w,\bfh,z_n)$.
The boundary conditions $v=0$, $\im(\bfh)=\bft$, $|z_n|=1$,
and $z_n(\rmi^k)=\rmi^k$ for $k=0,1,2$ force $u=\utsw$ for
suitable $\bfs,w$. This follows from a simple application
of the maximum and argument principles, cf.\ the proof of
Lemma~9 in~\cite{geze}. So for discs in this region, transversality is
obvious.

Holomorphic discs that leave this region have to pass through
a domain where $J$ may be chosen generic, which ensures
regularity for these discs, too.

No holomorphic disc can touch the $J$-convex boundary of $W$
at an interior point, so the boundary of the moduli space
is made up of discs $u$ with $u(\partial\D)\cap\partial W\neq\emptyset$.
We want to show that any such disc has to be standard. One may then
conclude that $\WW$ is a manifold (with boundary) of the
expected dimension $2n-1$, determined by the Maslov index.

The following slightly stronger statement will
be relevant for computing the degree of an evaluation map on
the moduli space later on. It is understood that $R$
has been chosen large.

\begin{lem}
\label{lem:non-standard}
Any non-standard disc in $\WW$ 
satisfies $w^2+|\bfh|^2\leq 8$ on~$\partial\D$.
\end{lem}

\begin{proof}
Consider the open subset $G\subset\D$ of points mapping into
that part of the handle where $|v|< R-\varepsilon$; this
set $G$ contains $\partial\D$. On the closure $\oG$
we consider the holomorphic components $g:=(v+\rmi w,\bfh)$
of~$u$; we suppress the component~$z_n$. Notice that for a standard disc
these components are constant.

If $\oG=\D$, then $u$ is standard by the argument preceding this lemma.

Let us now assume that $\oG$ is not all of~$\D$. On the
topological boundary $\partial G$ of $G$ in $\D$ we have
$|v|=R-\varepsilon$ or $w^2+|\bfh|^2/2+|z_n|^2/2=1$;
in the latter case this means $w^2+|\bfh|^2\leq 2$.
On $\partial\D$ we have $v=0$ and $\im(\bfh)=\bft$. This allows us to
extend the holomorphic map $g$ by Schwarz reflection
to the compact subset $S$ of the Riemann sphere given as the union
of $\oG$ and its reflected copy, cf.\ the proof of \cite[Lemma~11]{geze}.

Choose a point $s_0$ in $\partial\D$ with $|(w,\bfh)(s_0)|^2
=\bigl(\sqrt{2}+r_0\bigr)^2$. We want to show that $r_0\leq\sqrt{2}$.
The closed $2n$-ball $\oB_r$ of radius $r$ about the
point $g(s_0)$ does not intersect $g(\partial S)$ for $r<r_0$.
The monotonicity lemma~\cite[Theorem~II.1.3]{humm97} then says that the area
of $g(S)\cap\oB_r$ is bounded from below by $\pi r^2$.
So the area of $g(\oG)\cap\oB_r$ is bounded
from below by $\pi r^2/2$, and above by the total energy $\pi$ of~$u$.
It follows that $r\leq\sqrt{2}$. As a consequence,
the same inequality must hold for~$r_0$.
\end{proof}

\begin{rem}
\label{rem:non-standard}
With similar estimates one can show that in the part of the
handle where $|v|<R/2$, say, a neighbourhood of the upper boundary of the
handle contains standard discs only.
\end{rem}

An orientation of $\WW$ is determined as in \cite[Section~5.3]{geze}.
This will allow us to speak of the degree of the evaluation map
on~$\WW$.
\section{Proof of Theorem~\ref{thm:sphere}}
Given a contact manifold $(M,\xi=\ker\alpha)$ satisfying the
assumptions of Theorem~\ref{thm:sphere}, and a standard sphere
$S^{2n}\subset M$, we can build a symplectic
cobordism as in Section~\ref{subsection:cobordism}.
Write $\Sb\subset M_R$ for the belt sphere given as the intersection of
the upper boundary of the thick handle with the hyperplane
$\{v=0\}$ in~$\C^{n+1}$. This belt sphere can be described
explicitly in the form
\[ \Sb=\bigl\{ (v,w,\bfh,z_n)\co v=0,\; w^2+
\frac{1}{2}\bigl(|\bfh|^2+|z_n|^2\bigr)=(R-\delta)^2\bigr\},\]
where $R-\delta$ was the level near $R$ used to define the
thick handle in a neighbourhood of~$v=0$.
The pair $(M,S^{2n})$ is
diffeomorphic to $(M_R,\Sb)$, so it suffices to show that
$\Sb$ is homologically trivial in~$M_R$.

We are going to show that
the composition $f_0:=\mathrm{pr}_Y\circ\ev\co\WW\times\D\rightarrow M_R$
of the evaluation map
\[ \ev\co\WW\times\D\longrightarrow W,\;\;\; (u,z)\longmapsto u(z)\]
and the projection $\mathrm{pr}_Y$ of $W$ onto $M_R$ along the flow
lines of Liouville vector field~$Y$
sends the boundary $\partial(\WW\times\D)$ with degree~$1$
onto~$\Sb$. This implies the homological statement we want to prove,
and hence Theorem~\ref{thm:sphere}.

The boundary of $\WW\times\D$ is made up of the parts
$\partial\WW\times\D$ and $\WW\times\partial\D$.
The boundary $\partial\WW$ of the moduli space $\WW$ consists
of standard discs $\utsw$ whose boundary $\utsw(\partial\D)$ lies
in~$\Sb$, i.e.\ where $w^2+(|\bfh|^2+1)/2=(R-\delta)^2$.
The images $\utsw(\D)$ lie in the hyperplane
$\{v=0\}$. So when we project these images along the flow lines of
the Liouville vector field $Y$ to~$\Sb$, they
foliate a neighbourhood $S^{2n-2}\times D^2\subset\Sb$ of
\[ S^{2n-2}\times\{0\}=\Sb\cap\{z_n=0\}.\]

The part $\WW\times\partial\D$ of the boundary is mapped by the
evaluation map to points on some~$L^{\bft}$, where $v=0$ and $|z_n|=1$. Since
$Y$ is radial in $(w,\bfh,z_n)$-direction near $v=0$,
the points in $f_0(\WW\times\partial\D)$ satisfy $|z_n|\geq 1$.
In particular, this set is disjoint from a neighbourhood
of $S^{2n-2}\subset\Sb$.

It follows, as claimed, that the restriction
\[ f_0|_{\partial(\WW\times\D)}\co
\partial(\WW\times\D)\longrightarrow\Sb\]
is of degree~$1$.
\section{Proof of Theorem~\ref{thm:connected}}
\subsection{Deforming the evaluation map}
The first step in the proof of Theorem~\ref{thm:connected},
as in the proof of Theorem~\ref{thm:sphere},
is to construct a continuous map $\WW\times\D\rightarrow M_R$
sending $\partial(\WW\times\D)$ with degree~$1$ onto~$\Sb$. This time,
however, we should like to have more control over the behaviour
of the map near the boundary $\partial(\WW\times\D)$. This greatly
simplifies the subsequent topological arguments.

\vspace{1mm}

(1) Instead of projecting the whole image $\ev(\WW\times\D)$ along
the flow lines of $Y$ to~$M_R$, we first consider only the
image of the boundary $\partial(\WW\times\D)$ under~$\ev$.
For ease of notation, we write $\partial_+:=\ev(\partial\WW\times\D)$
and $\partial_-:=\ev(\WW\times\partial\D)$, so that the image in question
is given as the union $\partial_+\cup\partial_-$.

This image lies in the intersection of the
handle with the hyperplane $\{ v=0\}$. On that intersection,
we can define a map, homotopic to the identity, which
sends each segment of a flow line of $Y$ starting in a point of
$\partial_+$ and ending in a point of $\Sb$ to its
endpoint. Notice that under this requirement, points in
the set $\partial_+\cap\partial_-$, which by Lemma~\ref{lem:non-standard}
equals $\ev(\partial\WW\times\partial\D)$, do not move. So we may require
that the map do not move any points in $\partial_-$,
and no points outside a small neighbourhood of $\partial_+$
inside the full image of $\WW\times\D$ under~$\ev$.
Using $v$ as a homotopy parameter, we can extend this map to $W$,
equal to the identity outside a small neighbourhood of $\{ v=0\}$.
In this way we can ensure by Remark~\ref{rem:non-standard}
that points in the image of any
non-standard discs do not move. By postcomposing $\ev$ with this
map, we obtain a map $f_1\co\WW\times\D\rightarrow W$ whose
effect on $\partial_+$ is the same as that of the map $f_0$
considered in the proof of Theorem~\ref{thm:sphere}.

\vspace{1mm}

(2) Next we want to find a continuous map $f_2\co\WW\times\D
\rightarrow W$ that sends $\partial(\WW\times\D)$ with degree~$1$
onto~$\Sb$ (for the appropriate choice
of orientations). In contrast with the map $f_0$ considered
previously, we do not yet want any points in the
interior of $\WW\times\D$ to map to~$\Sb$ (or in fact~$M_R$).
Since $\WW$ is compact, we can choose a small annular neighbourhood
$A=[a,1]\times S^1$ of $\partial\D$ in $\D$ such that $f_1(\WW\times A)$,
when projected to $M_R$ along the flow lines of~$Y$,
is contained in a prescribed small neighbourhood of~$\Sb$.
For $r\in [a,1]$ and $\varphi\in S^1$ write $\gamma_{u,r,\varphi}$ for
the flow line of $Y$ through the point $f_1(u,r,\varphi)$,
reparametrised linearly such that $\gamma_{u,r,\varphi}(a)=
f_1(u,r,\varphi)$ and $\gamma_{u,r,\varphi}(1)$ is the point
on this flow line that lies in~$M_R$. Now define
$f_2\co\WW\times\D\rightarrow W$ by the requirements
\[ f_2=f_1\;\;\text{on}\;\;\WW\times\bigl(\D\setminus\Int(A)\bigr)\]
and
\[ f_2(u,r,\varphi)=\gamma_{u,r,\varphi}(r)\;\;\text{for}\;\;
(u,r,\varphi)\in\WW\times[a,1]\times S^1.\]
This map has the same effect on $\partial(\WW\times\D)$ as~$f_0$.

\vspace{1mm}

(3) The deformations we have considered thus far keep points in $\{v=0\}$
inside that hyperplane. Hence, if $u$ is a standard disc,
then $f_2(u,z)$ will be a point in $W\cap\{v=0\}$ for
any $z\in\D$, and a point in $M_R\cap\{v=0\}$ if and only
if $u\in\partial\WW$ or $z\in\partial\D$.

If $u$ is a non-standard disc, then $u(A)$ will be far away
from $M_R$ by Remark~\ref{rem:non-standard}. As a consequence,
$f_2(u\times A)$ will meet $M_R$
at a steep angle, i.e.\ the tangent vectors
have a small $v$-component compared to a large $(w,\bfh,z_n)$-component.

Hence, by a homotopy of $W$ that keeps $M_R$ fixed
and moves the interior of the handle near $\{v=0\}$
in the direction of negative~$v$, we obtain a new map
$f_3\co\WW\times\D\rightarrow W$
with the property that all points in $f_3(\Int(\WW)\times\Int(A))$
have negative $v$-component. On $\partial(\WW\times\D)$,
the map has not been changed.

\vspace{1mm}

(4) Now postcompose $f_3$ with the projection $\mathrm{pr}_Y$
of $W$ onto $M_R$ along the flow lines of~$Y$ to obtain a map
$f_4$ from $\WW\times\D$ to~$M_R$. Again, this
has no further effect on $\partial(\WW\times\D)$, but we are
now guaranteed that a neighbourhood of that boundary
inside $\WW\times\D$ is mapped by $f_4$ entirely to one side of the belt
sphere~$\Sb$.

\vspace{1mm}

(5) Finally, define
\[ f\co\bigl(\WW\times\D,\partial(\WW\times\D)\bigr)\longrightarrow
(M_R,\Sb)\]
as a smooth map homotopic to $f_4$ and transverse to~$\Sb$.
Then the preimage $f^{-1}(\Sb)$ will
consist of $\partial(\WW\times\D)$, which is mapped to
$\Sb$ with degree~$1$, and a collection of smooth hypersurfaces in the
interior of~$\WW\times\D$.

If $\WW$ is not connected, it will have one component
$\WW_{\partial}$ with boundary ---
the one containing the standard discs --- and some closed components.
The boundary component $\partial(\WW_{\partial}\times\D)$
of $\WW\times\D$ is the one that maps with degree~$1$ onto~$\Sb$; the
other boundary components map with degree zero to~$\Sb$, since their
image lies in the subset $\Sb\setminus S^{2n-2}\times\{0\}$.

\vspace{1mm}

For $u\in\WW$, the image point $u(1)$ lies in the $(2n-1)$-cell
\[ \bigl\{ (v+\rmi w,\bfs+\rmi\bft,z)\in\C\times\C^{n-1}\times\C\co
v=0,\; z=1,\; w^2+\frac{1}{2}\bigl(|\bfs|^2+|\bft|^2+1\bigr)\leq
(R-\delta)^2\bigr\}. \]
It follows that the image $f(\WW\times\{1\})$ is contained
in a $(2n-1)$-dimensional cell inside~$\Sb$.
\subsection{The mapping degree}
Write $M_i'$, $i=1,2$, for the closures of the connected components of
$M_R\setminus\Sb$. We choose the numbering such that
a collar neighbourhood of $\partial(\WW\times\D)$ is
sent to a neighbourhood of $\Sb=\partial M_1'$ in~$M_1'$.
Set $V=\WW_{\partial}\times\D$, where $\WW_{\partial}$ was defined
in point~(5) of the forgoing section,
and continue to write $f$ for~$f|_V$.
Set $V_i=f^{-1}(M_i')$. Then
\[ V=V_1\cup_{f^{-1}(\Sb)} V_2. \]
We orient $\Sb$ as the boundary of $M_1'$, and
$V$ such that its boundary is mapped by
$f$ with degree~$1$ onto~$\Sb$. Then
$f_i:=f|_{V_i}\co(V_i,\partial V_i)\rightarrow (M_i',\partial M_i')$
has a well-defined mapping degree, and these degrees can also
be computed as the degree of the restriction to the
boundary, $f_i|_{\partial V_i}\co \partial V_i\rightarrow\partial M_i'$.

\begin{lem}
\label{lem:degrees}
The mapping degrees of $f_1$ and $f_2$ satisfy
\[ \deg f_1-\deg f_2=1.\]
\end{lem}

\begin{proof}
We have $\partial M_1'=-\partial M_2'$ and $\partial V_1=-\partial V_2
\sqcup\partial V$, and $\partial V$ is mapped with degree~$1$
onto~$\partial M_1'$.
\end{proof}

In particular, the degree of at least one of $f_1,f_2$ is non-zero.
Together with the following lemma, this yields
information about the fundamental group. This lemma is
formulated for manifolds without boundary as an exercise
in \cite[p.~130]{hirs76}, so for the reader's convenience we include
the simple proof.

\begin{lem}
\label{lem:Hirsch}
Let $g\co (P,\partial P)\rightarrow (Q,\partial Q)$ be a map
between compact, oriented $m$-dimensional manifolds. We assume
that $Q$ and $\partial Q$ are connected, so that the mapping degree
$\deg g$ is defined and equal to $\deg g|_{\partial P}$. Likewise,
$P$ is assumed to be connected, and we define the
fundamental groups $\pi_1(Q),\pi_1(P)$ with respect to respective base
points $q_0\in Q$ and $p_0\in g^{-1}(q_0)$.

Then $g_*\bigl(\pi_1(P)\bigr)\subset\pi_1(Q)$
is a subgroup whose index divides $\deg g$. In particular:
\begin{itemize}
\item[(i)] If $\deg g\neq 0$ and $g_*\bigl(\pi_1(P)\bigr)$ is trivial,
then $\pi_1(Q)$ is a finite group.
\item[(ii)] If $\deg g=\pm 1$, then $g_*\co\pi_1(P)\rightarrow
\pi_1(Q)$ is onto.
\end{itemize}
\end{lem}

\begin{proof}
Let $\pi\co (\wQ,\partial\wQ,\tilde{q}_0)\rightarrow (Q,\partial Q,q_0)$
be the unique covering with characteristic subgroup
$\pi_*\bigl(\pi_1(\wQ)\bigr)$ equal to $g_*\bigl(\pi_1(P)\bigr)$. Then
$g$ lifts to a map $\tilde{g}\co (P,\partial P,p_0)\rightarrow
(\wQ,\partial\wQ,\tilde{q}_0)$.

The degrees of these maps satisfy $\deg g=\deg \pi\cdot\deg\tilde{g}$ ---
this is clear from the homological definition of the mapping degree ---
and $\deg\pi$ is the index of $\pi_*\bigl(\pi_1(\wQ)\bigr)=
g_*\bigl(\pi_1(P)\bigr)$ in $\pi_1(Q)$.
\end{proof}

We now want to apply this lemma to the maps $f,f_1,f_2$. First we
observe that one of the conditions in Lemma~\ref{lem:Hirsch}~(i)
is satisfied.

\begin{lem}
\label{lem:f-star}
For any choice of base points, the groups $f_*\bigl(\pi_1(V)\bigr)$ and
$(f_i)_*\bigl(\pi_1(V_i)\bigr)$, $i=1,2$, are trivial.
\end{lem}

\begin{proof}
The space $V=\WW_{\partial}\times\D$ retracts onto
$\WW_{\partial}\times\{1\}$, which is mapped by $f$ to a cell
inside~$\Sb$. This proves the statement for $f_*\bigl(\pi_1(V)\bigr)$.

Given a based loop $\gamma$ in~$V_i$, the loop
$f_i\circ\gamma$ is contractible in~$M_R$. Since $\pi_1(M_i')$
injects into $\pi_1(M_R)=\pi_1(M_1')*\pi_1(M_2')$, the loop
$f_i\circ\gamma$ must represent the trivial element in $\pi_1(M_i')$.
\end{proof}

The manifolds $V_1,V_2$ are not, in general, connected. Still, we can
speak of $\deg f_i$ as the sum of the degrees of the restriction
of $f_i$ to the components of~$V_i$.

\begin{prop}
\label{prop:deg}
The maps $f_i\co V_i\rightarrow M_i'$, $i=1,2$, have the following
properties:
\begin{itemize}
\item[(i)] If $\deg f_i\neq 0$, then $\pi_1(M_i')$ is finite.
\item[(ii)] If $\deg f_i=\pm 1$, then $\pi_1(M_i')$ is trivial.
\item[(iii)] If $\deg f_1\neq 0$, then one of $\pi_1(M_i')$ is trivial.
\item[(iv)] If both $\deg f_1$ and $\deg f_2$ are non-zero,
then $\pi_1(M)$ is finite.
\end{itemize}
\end{prop}

\begin{proof}
Lemma~\ref{lem:Hirsch} applies to each component of~$V_i$. If
$\deg f_i\neq 0$, then the degree on at least one of these
components is non-zero, and claim (i) follows with Lemma~\ref{lem:f-star}.

In fact, $\pi_1(M_i')$ must be a finite group whose order divides
the degree of each component of~$f_i$, and hence their sum $\deg f_i$.
This implies~(ii).

For (iii) we form a new space $\hat{V}$ by gluing $V$ to a copy of
$M_2'$ along their respective boundaries with the help
of the degree~$-1$ map $f|_{\partial V}\co\partial V\rightarrow
\partial M_2'$, i.e.\
\[ \hat{V}:= \bigl(V+M_2'\bigr)/x\sim f(x)\;\;\text{for}\;\;
x\in\partial V.\]
The map $f\co V\rightarrow M$ and the inclusion map $M_2'\rightarrow M$
induce a map $\hat{f}\co\hat{V}\rightarrow M$.

For computing the fundamental group and the homology of~$\hat{V}$,
it is best to think of this space as the union of three open subsets:
the interior of $V$, the interior of $M_2'$, and the mapping
cylinder of $f$ with open collars attached to its ends $\partial V$
and $\partial M_2'$. One then sees that $H_{2n+1}(\hat{V})\cong\Z$
and
\[ \pi_1(\hat{V})=\bigl(\pi_1(V)/\pi_1(\partial V)\bigr)*\pi_1(M_2').\]
In particular, $\hat{f}$ has a well-defined mapping degree,
and the proof of Lemma~\ref{lem:Hirsch} still goes through, even though
$\hat{V}$ may not be a manifold.

By Lemma~\ref{lem:f-star} we have $\hat{f}_*\bigl(\pi_1(\hat{V})\bigr)
=\pi_1(M_2')$, and by construction $\deg\hat{f}=\deg f_1$,
which by assumption is non-zero. With Lemma~\ref{lem:Hirsch}
we conclude that $\pi_1(M_2')$ is a subgroup of finite
index in $\pi_1(M)=\pi_1(M_1')*\pi_1(M_2')$. This can only happen
if $\pi_1(M_1')$ is trivial, or if $\pi_1(M_2')$ is trivial and
$\pi_1(M_1')$ is finite. (The finiteness of $\pi_1(M_1')$ is already
guaranteed by~(i).)

Statement (iv) follows immediately from (iii) and~(i).
\end{proof}
\subsection{The homology of the $M_i'$}
We now use the maps $f_i$ to derive homological information about
the manifolds~$M_i'$.

\begin{lem}
\label{lem:homology}
If $\deg f_i\neq 0$, then for $k\geq 1$ the homology groups $H_k(M_i';\Z)$
are finite and only contain $p$-torsion for primes $p$
that divide $\deg f_i$.
\end{lem}

\begin{proof}
Let $V_i^j$ be a component of $V_i$, and $f_i^j:=f_i|_{V_i^j}$.
For a field $\F$ of characteristic zero or coprime to $\deg f_i^j$,
the homomorphism $(f_i^j)_*\co H_k(V_i^j;\F)\rightarrow H_k(M_i';\F)$
is surjective for all~$k$, cf.\ \cite[Lemma~13]{geze}.

On the other hand, $f_*$ is the zero homomorphism in degrees
$k\geq 1$ by the argument used to prove Lemma~\ref{lem:f-star}.
This implies that the composition of $(f_i)_*$ with the
inclusion homomorphism $H_k(M_i;\F)\rightarrow H_k(M;\F)$
is the zero homomorphism. Since the inclusion homomorphism
is injective, $(f_i)_*$ must itself be trivial.

In combination, we conclude that $H_k(M_i';\F)$ is trivial
for $k\geq 1$ and characteristic of $\F$ equal to zero
or coprime to any $\deg f_i^j$.
\end{proof}
\subsection{Proof of Theorem~\ref{thm:connected}}
By Lemma~\ref{lem:degrees}, at least one $\deg f_i$ is
non-zero. If $M$ is simply connected and has torsion-free homology,
Lemma~\ref{lem:homology} implies that the corresponding $M_i'$ is a
simply connected homology ball with boundary diffeomorphic
to~$S^{2n}$, and hence a ball by the $h$-cobordism theorem,
cf.~\cite{geze}. 

If $\pi_1(M)$ is infinite, then by Proposition~\ref{prop:deg}~(iv),
one of $f_1,f_2$ has degree zero, and by Lemma~\ref{lem:degrees}
the other has degree $\pm 1$. From Proposition~\ref{prop:deg}~(ii)
and Lemma~\ref{lem:homology} we conclude that the corresponding
$M_i'$ is a simply connected homology ball, and hence a ball as before.

It remains to consider the case where $\pi_1(M)$ is finite but
non-trivial. This implies that $M_1'$, say, has finite fundamental group,
and $M_2'$ is simply connected. Then the universal covering
$\wtM$ is given by the universal covering
$\wtM_1'$ of $M_1'$, with a copy of $M_2'$ glued to
each of its $|\pi_1(M_1')|>1$ boundary components.
By Proposition~\ref{prop:deg} we have $\deg f_1\neq 1$, and hence
$\deg f_2\neq 0$. If $(\deg f_1,\deg f_2)=(0,-1)$, then $M_2'$
is a simply connected homology ball. So from now on we may assume
that both $\deg f_1$ and $\deg f_2$ are non-zero.

Choose a regular value $m_0\in\Sb\subset M$ for both $f$ and
$f|_{\partial V}$ as base point of~$M$, and a preimage $v_0\in
\partial V$ as base point of~$V$. Choose a preimage
$\wtm_0\in\wtM$ of $m_0$ under the covering
projection $\pi\co\wtM\rightarrow M$ as the base point
of~$\wtM$. Write $M_2^0$ for the copy of $M_2'$ in $\wtM$
that contains $\wtm_0$ in its boundary.

By Lemma~\ref{lem:f-star}, the map $f\co(V,v_0)\rightarrow (M,m_0)$ lifts
to a map
\[ \tilde{f}\co (V,v_0)\longrightarrow(\wtM,\wtm_0).\]
Write
\[ \tilde{f}_1\co (V_1,v_0)\longrightarrow (\wtM_1',\wtm_0)\]
for the restriction $\tilde{f}|_{V_1}$. Notice that if $V_1$ is
disconnected, this does not coincide with the lift of the map~$f_1$.
The boundary $\partial V$ is mapped with degree~$1$ onto~$\Sb$,
so the lifted map $\tilde{f}$ sends $\partial V$ with degree~$1$
onto the lifted copy of $\Sb$ containing the base point~$\wtm_0$,
i.e.\ the boundary of $M_2^0$.

We have $\tilde{f}_1(\partial V_1)\subset\partial\wtM_1'$,
so the mapping degree of $\tilde{f}_1$ is defined, and
\[ \deg f_1=|\pi_1(M_1')|\cdot\deg\tilde{f}_1.\]
Since $f$ is transverse to~$\Sb$, any regular value for
$f|_{\partial V_1}$ is also regular for~$f$. So by Sard's theorem we can
choose a regular value for $\tilde{f}_1$ on each boundary
component, and $\deg\tilde{f}_1$ equals the number of
preimages of any of these values.

Now consider the restriction of $\tilde{f}$ to~$V_2$. This map sends
each component of $V_2$ to one of the copies of $M_2'$ in~$\wtM$,
with boundary mapping to boundary, so the mapping degree
is defined. Choose a copy $M_2^1\neq M_2^0$
of $M_2'$ in~$\wtM$, and write $\tilde{f}_2^0,\tilde{f}_2^1$
for the restriction of $\tilde{f}$ to the preimage
of $M_2^0$ or $M_2^1$, respectively. A point in
$\partial M_2^1\subset\partial\wtM_1'$ has preimages only
in the interior of~$V$; a regular value in $\partial M_2^0$
also has a single preimage (counted with signs) in~$\partial V$.
Since the total number of preimages of any regular value in
$\wtM_1'$ equals $\deg\tilde{f}_1$, we have
\[ \deg\tilde{f}_2^1=\deg\tilde{f}_1,\;\;\;
\deg\tilde{f}_2^0=\deg\tilde{f}_1-1.\]

Now repeat the argument of Lemma~\ref{lem:homology} with the
map
\[ \pi\circ\tilde{f}_2^1\co \bigl(\tilde{f}_2\bigr)^{-1}(M_2^1)
\longrightarrow M_2'.\]
This shows that $H_k(M_2')$ can only contain $p$-torsion
for primes $p$ that divide
\[ \deg(\pi\circ\tilde{f}_2^1)=\deg\tilde{f}_1,\]
and hence $\deg f_1$. But $p$ must also divide
$\deg f_2=\deg f_1-1$. Once again, we deduce that $M_2'$
is a simply connected homology ball.

This concludes the proof of Theorem~\ref{thm:connected}.

%\vspace{2mm}

\begin{rem}
Certain conclusions can be drawn in the simply connected case
even in the presence of torsion. We illustrate this by an example.
The $5$-dimensional Brieskorn manifold $\Sigma(2,3,3,3)$ is
a simply connected spin manifold with $H_2\cong\Z_2\oplus\Z_2$,
see \cite{cbt76} or~\cite{koer08}; these data characterise
the manifold up to diffeomorphism. Brieskorn manifolds carry
natural contact structures, cf.\ \cite[Section~7.1]{geig08}.
Let $M_1,M_2$ be two copies of this manifold, equipped with
arbitrary contact structures $\xi_1,\xi_2$, and let
$(M,\xi)$ be their contact connected sum. Then any contact
form defining $\xi$ must
have a contractible periodic Reeb orbit, for otherwise we could
speak of the evaluation maps $f_1,f_2$, and the prime $2$ would
divide both $\deg f_1$ and $\deg f_2=\deg f_1-1$.
\end{rem}

\begin{ack}
We thank Peter Albers, Stefan Friedl and Stefan Suhr for
useful conversations.
\end{ack}

\end{document}